\documentclass[reqno]{amsproc}

\usepackage{graphicx}

\usepackage{hyperref}

\hypersetup{
colorlinks=true,
linkcolor=blue,
anchorcolor=blue,
citecolor=blue
}

\usepackage{amssymb}
\usepackage{amsfonts}
\usepackage{amsmath}
    
\newtheorem{theorem}{Theorem}[section]
\newtheorem{lemma}[theorem]{Lemma}
\newtheorem{proposition}[theorem]{Proposition}
\newtheorem{corollary}[theorem]{Corollary}

\theoremstyle{definition}
\newtheorem{definition}[theorem]{Definition}
\newtheorem{example}[theorem]{Example}

\newtheorem{remark}[theorem]{Remark}

\numberwithin{equation}{section}

\usepackage{ifdraft}
\usepackage[abbrev,backrefs]{amsrefs}
\usepackage{todonotes}
\newcommand{\comment}[1]{\ifdraft{\todo[size=\small, color=orange!35]{#1}}{}}

\usepackage{mathtools}
\usepackage{enumerate}

\newcommand{\N}{\mathbb{N}}
\newcommand{\R}{\mathbb{R}}

\newcommand{\dif}{\,\mathrm{d}}
\DeclareMathOperator{\capt}{Cap}
\DeclareMathOperator{\supp}{supp}
\newcommand{\Class}[1]{[#1]}

\makeatletter
\newcommand\my@saved@item{}
\newcommand\mydescriptionitem{}
\def\mydescriptionitem[#1]{%
  \my@saved@item[{\csname phantomsection\endcsname#1}]%
  \def\@currentlabel{\unexpanded{\unexpanded{#1}}}%
}%
\newenvironment{mydescription}%
               {%
                  \let\my@saved@item=\@item
                  \let\@item=\mydescriptionitem
                  \description
               }%
               {\csname enddescription\endcsname}
\makeatother

\usepackage{constants}

\begin{document}

\title{Some Remarks on Capacitary Integrals and Measure Theory}

\author{Augusto C. Ponce}
\address{Université catholique de Louvain,
Institut de recherche en mathématique et physique,
 Chemin du cyclotron 2, L7.01.02, 1348 Louvain-la-Neuve,
Belgium}
\email{Augusto.Ponce@uclouvain.be}

\author{Daniel Spector}
\address{National Taiwan Normal University,
Department of Mathematics,
No. 88, Section 4, Tingzhou Road, Wenshan District,
 Taipei City, Taiwan 116, R.O.C.}
\email{spectda@protonmail.com}

\thanks{D.~Spector is supported by the Taiwan Ministry of Science and Technology under research grant number 110-2115-M-003-020-MY3 and the Taiwan Ministry of Education under the Yushan Fellow Program.}

\subjclass[2010]{Primary ; Secondary }


\dedicatory{Dedicated to the memory of David R. Adams}


\begin{abstract}
We present results for Choquet integrals with minimal assumptions on the monotone set function through which they are defined. 
They include the equivalence of sublinearity and strong subadditivity independent of regularity assumptions on the capacity, as well as various forms of standard measure theoretic convergence theorems for these non-additive integrals, e.g. Fatou's lemma and Lebesgue's dominated convergence theorem.   
\end{abstract}

\maketitle
\tableofcontents


%
%
%
%
%
%
%

\section{Introduction}

The oeuvre of David R.\@~Adams has had a profound influence on the study of Sobolev inequalities, including results early in his career on trace inequalities \cites{A1,A7,A11}, numerous papers over the years concerning potentials \cites{A4,A5,A8,A9,A12,A13,A19-a}, and of special interest in this paper, his body of work on capacities and Choquet integrals \cites{A2,A3,A10,A14,A16,A17,A18,A19,A23-1998,A24,A25,A26}.  
That one should be interested in the study of Choquet integration is clear from the consideration of strong forms of the Sobolev inequality, namely V. Maz'ya's capacitary inequalities \cites{Mazya1,Mazya1a,Mazya2} and their various extensions \cites{A14,Dahlberg:1979,A19,C6,Ponce-Spector:2018, Ponce-Spector:2020}, as it is precisely in these improvements to typical Lebesgue or Lorentz inequalities that these integrals make an appearance.  
These inequalities give usual compactness results, though are strong enough even to provide information about fine properties of functions, and therefore motivate the need for as robust as possible of a framework of Choquet integration which contains these main capacitary inequalities as examples.  The work we reference of D.\@~R.\@~Adams provides a number of results in this direction, most notably his survey \cites{A23-1998}.

The starting place of D.\@~R.\@~Adams is the treatise of G.\@~Choquet \cites{Choquet:1954}, who developed a theory of integration with respect to monotone, countably subadditive set functions with  additional regularity assumptions:  
We say that $H:\mathcal{P}(\mathbb{R}^d) \to [0,\infty]$,  defined on the class  $\mathcal{P}(\mathbb{R}^d)$ of all subsets of $\mathbb{R}^d$, is a capacity in the sense of Choquet whenever it satisfies the conditions
\begin{mydescription}
\item[{\sc empty set}]  \label{null} $H(\emptyset)=0$;
\item[{\sc monotonicity}]  \label{monotone} If $E \subset F \subset \R^d$, then $H(E)\leq H(F)$;
\item[{\sc countable subadditivity}] \label{countable_subadditivity} For every sequence of sets $E_n \subset \R^d$,
\begin{align*}
H\Bigl(\bigcup\limits_{n=0}^\infty E_n \Bigr) &\leq \sum_{n=0}^\infty H(E_n);
\end{align*}
\item[{\sc outer regularity}] \label{continuity_from_above}   For every non-increasing sequence of compact subsets $K_n \subset \R^d$,
\begin{align*}
H\Bigl(\bigcap\limits_{n=0}^\infty K_n \Bigr)= \lim_{n \to \infty} H(K_n);
\end{align*}
\item[{\sc inner regularity}]  \label{continuity_from_below}
For every nondecreasing sequence of sets $E_n \subset \R^d$,
\begin{align*}
H\Bigl(\bigcup\limits_{n=0}^\infty E_n \Bigr) = \lim_{n \to \infty} H(E_n).
\end{align*}
\end{mydescription}

Given a set function $H:\mathcal{P}(\mathbb{R}^d) \to [0,\infty]$ that merely satisfies \ref{monotone}, one can define the Choquet integral with respect to $H$ of any function $f : \R^{d} \to [0, \infty]$ as
\begin{align}\label{choquet_def}
\int  f \dif H\vcentcolon= \int_0^\infty H(\{ f >t\}) \dif t,
\end{align}
where the right-hand side is understood as the Lebesgue integral of the non-increasing function
\begin{align*}
t \in (0, \infty) \longmapsto H(\{ f > t\}).
\end{align*}
Such an integral has a number of desirable properties, for example \eqref{choquet_def} is positively $1$-homogeneous and monotone.  
Moreover, one may replace the sets $\{ f > t\}$ with $\{ f \geq t\}$ and obtain the same value for the integral.  
However, one consequence of the choice to integrate outside the framework of Measure Theory is that this integral need not be linear, and in fact may not even be sublinear.  Indeed, G.\@~Choquet \cite{Choquet:1954}*{54.2 on p.~289} established a necessary and sufficient condition on $H$ that the integral be sublinear:
\begin{theorem}\label{choquet}
Let $H$ be a capacity in the sense of Choquet.  Then, the Choquet integral \eqref{choquet_def} is sublinear if and only if $H$ is strongly subadditive.
\end{theorem}
Here we recall the notion of
\begin{mydescription}
\item[{\sc strong subadditivity}] \label{strong_subadditive} 
For every sets $E, F \subset \R^d$,
\comment{We need to choose the label accordingly.}
\begin{align*}
H(E\cap F)+H(E\cup F) \leq H(E)+H(F).
\end{align*}
\end{mydescription}

The sublinearity of the integral implies one has a triangle inequality, from which H\"older's and Minkowski's inequalities follow from usual convexity arguments.  These inequalities in turn serve as a basis for the study of a family of Banach spaces of functions $L^p(H)$, those suitably regular functions whose $p$th power has finite Choquet integral.  Here typical questions have concerned the boundedness of maximal functions \cites{A19,C8,STW}, characterizations of the topological duals \cites{A19,C10}, and interpolation theory \cites{CCM,CJ,CMS}.  
The assumption one has a capacity in the sense of Choquet ensures that even without the full strength of results from Measure Theory one has a number of useful tools, e.g.\@ Fatou's lemma (which follows from \ref{monotone} and \ref{continuity_from_below}, see \cite{CMS}*{Theorem~1 on pp.~98--99}):
\begin{align}\label{Choquet_Fatou}
\int \liminf\limits_{n \to \infty}{f_n} \dif H
\le \liminf_{n \to \infty}{\int f_n \dif H},
\end{align}
for every sequence of functions $f_n: \R^d \to [0,\infty]$.  As a result the Choquet integral built on a strongly subadditive capacity in the sense of Choquet enjoys countable sublinearity:  For every sequence of functions $f_n: \R^d \to [0,\infty]$ one has
\begin{align}\label{Choquet_Fatou_2}
\int \sum_{n=0}^\infty f_n \dif H
\le \sum_{n=0}^\infty {\int f_n \dif H},
\end{align}
and a Fatou-type lemma that is often appealed to (see e.g. the argument on p.~123 of \cite{A19}:  If $f_n \to f$ locally in $L^1(\mathbb{R}^d)$ (or pointwise almost everywhere), then
\begin{align*}
\int \mathcal{M}f \dif H
\le \liminf_{n \to \infty}{\int \mathcal{M}f_n \dif H},
\end{align*}
where 
\begin{align*}
\mathcal{M}f(x) := \sup_{r>0} \frac{1}{r^d}\int_{B_r(x)} |f(y)|\dif y
\end{align*}
is the Hardy-Littlewood maximal function.  These results are a small sample of the theory of Choquet integration developed and recorded for capacities in the sense of Choquet, and we refer the reader to \cites{A19,Choquet:1954,CMS} for further details.

Unfortunately, in practice the capacities that arise in various inequalities \cites{Mazya1,Mazya1a,Mazya2,A14,Dahlberg:1979,A19,C6} may fail to satisfy \ref{continuity_from_below} or \ref{continuity_from_above}, a notable example being the Hausdorff content or its dyadic version, see \cite{HH23} and Example~\ref{content_example} below. 
It is natural then that one address the necessity of these regularity assumptions in the resulting theory of Choquet integration.  
This program was initiated by Adams, who typically did not require that the set functions under consideration be capacities in the sense of Choquet, often with the initial assumptions of only \ref{null}, \ref{monotone}, and \ref{countable_subadditivity}. 
He referred to these objects as capacities in the sense of N.~Meyers, though the reader may also recognize these are the defining properties of an outer measure.  To these he then added a continuity assumption, \ref{continuity_from_above} or \ref{continuity_from_below}, either of which is sufficient to obtain Choquet's characterization of sublinearity of the integral (see Anger's paper \cite{Anger:1977} for the proof assuming \ref{continuity_from_above} or Saito, Tanaka, and Watanabe's paper \cite{STW}*{Proposition~3.2} for the proof assuming \ref{continuity_from_below}).  The first observation of this paper is that neither assumption is necessary, that one has the following characterization independent of regularity assumptions.

\begin{theorem}\label{sublinear}
Suppose that $H$ satisfies \ref{monotone}.  
Then, the Choquet integral \eqref{choquet_def} is sublinear if and only if $H$ is strongly subadditive.
\end{theorem}

Here is a simple example for which Theorem~\ref{sublinear} applies but not Theorem~\ref{choquet}:
\begin{example}
\label{exampleCapacityFiniteSets}
Let $H$ be the set function defined by
\[   H(A)\vcentcolon= 
\begin{cases}
      0 & \text{if \(A\) is finite,} \\
      1 & \text{if \(A\) is infinite.}
\end{cases} 
\]
Then, $H$ is strongly subadditive and satisfies \ref{null} and \ref{monotone}, but does not satisfy \ref{countable_subadditivity}, \ref{continuity_from_above} or \ref{continuity_from_below}.
\end{example}

A second pertinent example is the dyadic Hausdorff content for which Yang and Yuan \cite{Yang-Yuan:2008} observed the following

\begin{example}\label{content_example}
Let $0 < \beta < d$ and let $H=\tilde{\mathcal{H}}^\beta_\infty$ be the set function defined by
\begin{align}\label{content}
\tilde{\mathcal{H}}^\beta_\infty(E)\vcentcolon= \inf \left\{ \sum_{n=0}^\infty \ell(Q_n)^\beta : 
Q_n \text{ is a dyadic cube and  } E \subset \bigcup_{n=0}^\infty Q_n  \right\},
\end{align}
where \(\ell(Q_n)\) denotes the side-length of \(Q_n\).
Then, $H$ is strongly subadditive and satisfies \ref{null}, \ref{monotone}, \ref{countable_subadditivity}, and \ref{continuity_from_below}, but not \ref{continuity_from_above} for $\beta \le d-1$.
\end{example}

The initial impetus for this work was the question of the validity of Theorem~\ref{sublinear}, though after obtaining a proof we discovered in our broader literature review that this was  known to the community of non-additive Measure Theory \cite{Denneberg}*{Chapter~6}.  
As it seems to have not been referenced in the results after Adams, we give the proof below for the convenience of the reader, the idea of which is as follows:  First, one proves an algebraic result, which amounts to sublinearity for finite sums of characteristic functions (see e.g., the argument at the top of p.~249 of \cite{Anger:1977}, the argument on pp.~766--768 of \cite{STW}, or Proposition~\ref{algebraic} below);  Second, one argues the general case by approximation.  
In the papers \cites{Anger:1977,STW} this is performed invoking either \ref{continuity_from_above} or \ref{continuity_from_below} to justify the limit, though as we show below it can be done by using only \ref{monotone} and properties of the Lebesgue integral in \eqref{choquet_def}.

That the Choquet integral is sublinear without any regularity assumptions is perhaps surprising to the community working on capacitiary inequalities in the spirit of Adams, e.g.\@ \cites{Cap1, Cap2, C1,C2,C3,C4,C5,C6,C7,C8,C9,C10,C11, STW, Ponce-Spector:2018, Ponce-Spector:2020, Spector:PM, Chen-Spector}, and suggests that it should be interesting to understand what other aspects of the theory of Choquet integration relies on these regularity assumptions and in what areas it can be dispensed.  In this paper we take up this question as pertains to analogues of measure theoretic results, in particular Fatou's lemma and Lebesgue's dominated convergence theorem, as well as functional analysis results concerning spaces of functions with finite capacitary integral.  Our results show that while one does not need regularity of the capacity, in one way or another regularity must make an appearance in order to obtain such results.  In particular, if one assumes regularity on any of the capacity, the mode of convergence, or the functions involved, then it is possible to obtain analogues of \eqref{Choquet_Fatou} and \eqref{Choquet_Fatou_2}, see Figure~\ref{figureTrinity}.

\begin{figure}
\centering
\includegraphics[width = 0.5 \textwidth]{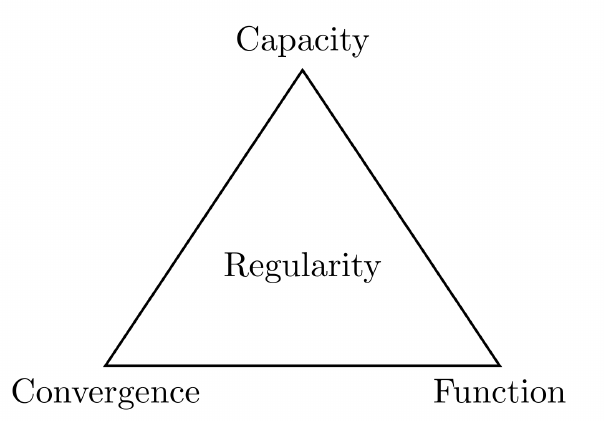}
\caption{Trinity of assumptions}
\label{figureTrinity}
\end{figure}

The plan of the paper is to make precise the sketch presented in Figure~\ref{figureTrinity}, as well as to meticulously develop the assumptions that lead to various results of the measure theoretic and functional analytic aspects of Choquet integration.  
Our results are roughly organized in terms of increasing assumptions on the capacity as one proceeds through the paper.  With this framework, in Section~\ref{quconvergence}, we rely on a strong form of pointwise convergence with respect to $H$, namely quasi-uniform convergence, which is sufficient to obtain versions of Fatou's lemma and Lebesgue's dominated convergence theorem for Choquet integrals with minimal assumptions on the capacity.  
Conversely, we show that convergence of a sequence of functions with respect to the Choquet integral implies that the sequence has this strong convergence property, the idea of which follows the typical completeness argument for the functional space $L^1(H)$.  
In Section~\ref{sectionApproximation}, we recall the notion of quasicontinuity.  
When one imposes certain additional conditions on $H$, we show that quasicontinuous functions admit approximation with respect to the Choquet integral by functions in \(C_c(\R^d)\), the class of  continuous functions that are compactly supported in \(\R^d\).  
In Section~\ref{proofs}, we prove Theorem~\ref{sublinear} on the sublinearity of the Choquet integral.  
In Section~\ref{lowersemicontinuity} we show how one can relax the notion of convergence provided one works within the class of quasicontinuous functions.  
This relies on a classical application of the Hahn-Banach Theorem that realizes the Choquet integral as a supremum over Lebesgue integrals with respect to locally finite measures underneath the capacity.  
In Section~\ref{Lone}, we introduce the Banach space $L^1(H)$ as the set of equivalence classes of quasicontinuous functions for which the Choquet integral is finite.  
The results in the preceding sections are then shown to imply standard results concerning this space such as completeness, density of \(C_c(\R^d)\), after which we discuss Fatou's lemma in this context, namely we provide closure properties that guarantee that the limit of a sequence of functions remains in \(L^1(H)\).  
For ease of reference we provide proofs of all of our assertions.  
This paper is an homage to David R.\@~Adams' work, for which we are extremely grateful, and is dedicated to his memory.

\section{Quasi-uniform convergence}\label{quconvergence}

In this section, we present some convergence properties for the Choquet integral with respect to a monotone set function \(H : \mathcal{P}(\R^d) \to [0, \infty]\) under quasi-uniform convergence of the integrand.

\begin{definition}
Let \((f_n)_{n \in \N}\) be a sequence of functions \(f_n : \R^d \to [-\infty, \infty]\).
We say that \((f_n)_{n \in \N}\) converges quasi-uniformly to \(f : \R^d \to [-\infty, \infty]\) whenever, for each \(\epsilon > 0\), there exists \(E \subset \R^d\) such that \(H(E) \le \epsilon\) and \((f_n)_{n \in \N}\) is finite and converges uniformly to \(f\) in \(\R^d \setminus E\).
We denote
\[
f_n \to f 
\quad \text{q.u.}
\]
\end{definition}

For simplicity, we omit the dependence of \(H\).
In the case where \(H\) is a measure, almost everywhere pointwise convergence implies quasi-uniform convergence on sets of finite measure.
However, this is not true for capacities in general:

\begin{example}
Take a capacity \(H\) such that \(H(\partial B_r) \ge \eta\) for every \(1 < r < 2\) and some fixed \(\eta > 0\), where \(B_r \vcentcolon= B_r(0)\) is the ball of radius \(r\) centered at \(0\).
Such is the case with the Hausdorff content \(\mathcal{H}^\beta_\infty\) with \(\beta \le d - 1\) or the Sobolev capacity \(\capt_{1, p}\) for \(1 \le p < d\), see \cites{AH,P} for their definitions and further properties.
If \((f_n)_{n \in \N_*}\) is any sequence of functions in \(\R^d\) such that \(f_n \ge 1\) on \(\partial B_{1 + 2/n}\) with \(f_n\) supported in \(B_{1 + 3/n} \setminus B_{1 + 1/n}\), then \(f_n \to 0\) pointwise in \(\R^d\), but
not quasi-uniformly since \(H(\partial B_{1 + 2/n}) \ge \eta\) for every \(n \in \N_*\).
\end{example}

In this section we rely mostly on subadditivity involving finitely many sets:
\begin{mydescription}
\item[{\sc finite subadditivity}] \label{finite_subadditivity}  For every \(E, F \subset \R^d\),
    \[
    H(E \cup F)
    \le H(E) + H(F).
    \]
\end{mydescription}

The following version of Fatou's lemma for quasi-uniform convergence then holds:

\begin{proposition}
    \label{propositionUniformFatou}
    Suppose that \(H\) satisfies \ref{monotone} and \ref{finite_subadditivity}.
    If \((f_n)_{n \in \N}\) is a sequence of nonnegative functions in \(\R^d\) such that \(f_n \to f\) q.u., then
    \[
    \int f \dif H
    \le \liminf_{n \to \infty}{\int f_n \dif H}.
    \]
\end{proposition}

\begin{proof}
    Given \(\epsilon > 0\), there exists \(E \subset \R^d\) such that \(H(E) \le \epsilon\) and \(f_n \to f\) uniformly in \(\R^d \setminus E\).
    Given \(\eta > 0\), take \(N \in \N\) such that, for \(n \ge N\), 
    \[
    |f_n - f| \le \eta 
    \quad \text{in \(\R^d \setminus E\).}
    \]
    We then have
    \[
    f \le f_n + |f_n - f| \le f_n + \eta 
    \quad \text{in \(\R^d \setminus E\).}
    \]
    Hence, for every \(t > 0\),
    \[
    \{f > t + \eta\} \setminus E \subset \{f_n > t\}
    \]
    which implies that
    \[
    \{f > t + \eta\} \subset \{f_n > t\} \cup E.
    \]
    By monotonicity and finite subadditivity of \(H\), we get
    \[
    H(\{f > t + \eta\}) 
    \le H(\{f_n > t\}) + H(E).
    \]
    Fix \(k \in \N\).
    Integrating with respect to \(t\) over the interval \((0, k)\),
    \[
    \begin{split}
    \int_\eta^{k + \eta} H(\{f > s\}) \dif s
    & = \int_0^{k} H(\{f > t + \eta\}) \dif t\\
    & \le \int_0^{k} H(\{f_n > t\}) \dif t + H(E) k
    \le \int f_n \dif H + \epsilon k.
    \end{split}
    \]
    This estimate holds for every \(n \ge N\).
    Letting \(n \to \infty\), we get
    \[
    \int_\eta^{k + \eta} H(\{f > s\}) \dif s
    \le \liminf_{n \to \infty}{\int f_n \dif H} + \epsilon k.
    \]
    As \(\epsilon \to 0\), we deduce that
    \[
    \int_\eta^{k + \eta} H(\{f > s\}) \dif s
    \le \liminf_{n \to \infty}{\int f_n \dif H}.
    \]
    To conclude we let \(k \to \infty\) and \(\eta \to 0\).
    We then get by Fatou's lemma for the Lebesgue measure,
    \[
    \int f \dif H
    = \int_0^{\infty} H(\{f > s\}) \dif s
    \le \liminf_{n \to \infty}{\int f_n \dif H}.
    \qedhere
    \]
\end{proof}

We now show the following version of the Dominated Convergence Theorem:

\begin{proposition}
    Suppose that \(H\) satisfies \ref{monotone} and \ref{finite_subadditivity}.
    If \((f_n)_{n \in \N}\) is a sequence of real-valued functions in \(\R^d\) such that \(f_n \to f\) q.u.\@ and if there exists \(F : \R^d \to [0, \infty]\) such that \(\int F \dif H < \infty\) and \(|f_n| \le F\) in \(\R^d\) for every \(n \in \N\), then  \(\int |f| \dif H < \infty\) and
    \[
    \lim_{n \to \infty}{\int |f_n - f| \dif H} = 0.
    \]
\end{proposition}

\begin{proof}
    Since \(|f_n| \to |f|\) q.u.\@ and \(|f_n| \le F\) in \(\R^d\), by Proposition~\ref{propositionUniformFatou} and by monotonicity of the Choquet integral we have
    \[
    \int |f| \dif H
    \le \liminf_{n \to \infty}{\int |f_n| \dif H}
    \le \int F \dif H
    < \infty.
    \]
    Given \(\epsilon > 0\), take \(E \subset \R^d\) with \(H(E) \le \epsilon\) such that \((f_n)_{n \in \N}\) converges uniformly to \(f\) in \(\R^d \setminus E\). 
    Then, given \(\eta > 0\), let \(N \in \N\) be such that, for every \(n \ge N\), \(|f_n - f| \le \eta\) in \(\R^d \setminus E\).
    Thus, for \(t > \eta\) and \(n \ge N\), we have 
    \[
    \{|f_n - f| > t\} \subset E,
    \]
    whence, by monotonicity of \(H\),
    \begin{equation}
    \label{eq425}
    H(\{|f_n - f| > t\}) 
    \le H(E) \le \epsilon.
    \end{equation}
    Also, for any \(t > 0\),
    \begin{equation}
    \label{eq431}
    \{|f_n - f| > t\} \subset \{F + |f| \ge t\}.
    \end{equation}
    It follows from \eqref{eq425} and \eqref{eq431} that, for any \(k > \eta\) and \(n \ge N\),
    \[
    \begin{split}
    \int |f_n - f| \dif H
    & = \int_0^\eta
    + \int_\eta^k 
    + \int_k^\infty H(\{|f_n - f| > t\}) \dif t\\
    & \le \int_0^\eta H(\{F + |f| > t\}) \dif t
    + k \epsilon 
    + \int_k^\infty H(\{F + |f| > t\}) \dif t.
    \end{split}
    \]
As \(n \to \infty\),
\begin{multline}
\label{eq448}
\limsup_{n \to \infty}{\int |f_n - f| \dif H}\\
\le \int_0^\eta H(\{F + |f| > t\}) \dif t
    + k \epsilon 
    + \int_k^\infty H(\{F + |f| > t\}) \dif t.
\end{multline}
By finite subadditivity of \(H\),
\[
H(\{F + |f| > t\})
\le H(\{F \ge t/2\}) + H(\{|f| \ge t/2\}).
\]
Since both \(F\) and \(|f|\) have finite Choquet integrals, we have the conclusion by letting \(\epsilon, \eta \to 0\) and then \(k \to \infty\) in \eqref{eq448}.
\end{proof}

The proof of the partial converse of the Dominated Convergence Theorem involves countable subadditivity:

\begin{proposition}
    \label{propositionDominatedConvergenceConverse}
    Suppose that \(H\) satisfies \ref{monotone} and \ref{countable_subadditivity}.
    Let \((f_n)_{n \in \N}\) be a sequence of real-valued functions in \(\R^d\) such that
    \[
    \int |f_n - f| \dif H \le \frac{1}{4^n}
    \quad \text{for every \(n \in \N\),}
    \]
    where \(f : \R^d \to \R\) satisfies \(\int |f| \dif H < \infty\). 
    Then, \(f_n \to f\) q.u.\@ and there exists \(F : \R^d \to [0, \infty]\) such that \(\int F \dif H < \infty\) and \(|f_n| \le F\) in \(\R^d\) for every \(n \in \N\).
\end{proposition}

\begin{proof}
By monotonicity of \(H\) one has an analogue of Chebyshev's inequality, 
\begin{equation}
\label{eq653}
\frac{1}{2^n} H\bigl( \bigl\{|f_{n} - f| > 1/2^n \bigr\} \bigr)
 \le \int |f_{n} - f| \dif H
 \le \frac{1}{4^n}.
\end{equation}
For each \(k \in \N\), denoting
\[
A_k 
\vcentcolon= \bigcup_{n = k}^\infty \bigl\{|f_{n} - f| > 1/2^n \bigr\},
\]
then, by countable subadditivity of \(H\) and  \eqref{eq653},
\begin{equation}
\label{eq676}
H(A_k) 
\le \sum_{n = k}^{\infty}{H\bigl( \bigl\{|f_{n} - f| > 1/2^n \bigr\} \bigr)}
\le \sum_{n = k}^{\infty}{\frac{1}{2^{n}}}
\le \frac{1}{2^{k - 1}}.
\end{equation}

Since the sequence \((1/2^n)_{n \in \N}\) is summable, by the Weierstrass M-test the series \(\sum\limits_{n = k}^\infty{|f_{n} - f|}\) converges uniformly in \(\R^d \setminus A_k\) for every \(k \in \N\), and then so does the sequence \((f_n)_{n \in \N}\).
We deduce from \eqref{eq676} that \((f_n)_{n \in \N}\) converges quasi-uniformly to \(f\) in \(\R^d\).

To conclude, it suffices to verify that \(F \vcentcolon= |f| + \sum\limits_{n = 0}^\infty{|f_{n} - f|}\) has finite Choquet integral.
We first observe that, by quasi-sublinearity of the Choquet integral, one has that for every two functions \(g\) and \(h\),
\begin{equation}
\label{eqQuasiSublinearity}
\int |g + h| \dif H
\le 2 \int |g| \dif H + 2 \int |h| \dif H.
\end{equation}
Iterating this inequality, for every \(j \in \N\) one gets
\[
\int \sum_{n = 0}^j{|f_{n} - f|} \dif H
\le \sum_{n = 0}^j{ 2^{n + 1} \int |f_{n} - f| \dif H}
\le \sum_{n = 0}^j{\frac{1}{2^{n - 1}}}
\le 4.
\]
By quasi-sublinearity of the Choquet integral and quasi-uniform convergence of the series \(\sum\limits_{n = 0}^\infty{|f_{n} - f|}\), we deduce from Proposition~\ref{propositionUniformFatou} that
\[
\int F \dif H \le 2 \int |f| \dif H + 8 < \infty.
\qedhere
\]
\end{proof}

The need for \ref{countable_subadditivity} in the statement of Proposition~\ref{propositionDominatedConvergenceConverse} can be seen in the following

\begin{example}
    Take a sequence \((a_n)_{n \in \N}\) of distinct points in \(\R^d\) and, for each \(n \in \N\), let \(f_n(a_n) = 1\) and \(f_n(x) = 0\) for \(x \ne a_n\).
    If \(H\) is the capacity given by Example~\ref{exampleCapacityFiniteSets}, then \(\int |f_n - 0| \dif H = H(\{a_n\}) = 0\) for each \(n\) but \((f_n)_{n \in \N}\) does not converge q.u.\@ to \(0\) since, for every \(j \in \N\), we have \(H\bigl(\bigcup\limits_{n \ge j}\{a_n\}\bigr) = 1\).
\end{example}

\section{Approximation of quasicontinuous functions}
\label{sectionApproximation}

We investigate in this section the question of approximation of quasicontinuous functions with finite Choquet integral by sequences of continuous functions.

\begin{definition}
A function \(f : \R^d \to [-\infty, \infty]\) is quasicontinuous whenever, for each \(\epsilon > 0\), there exists an open set \(\omega \subset \R^d\) such that \(H(\omega) \le \epsilon\) and \(f|_{\R^d \setminus \omega}\) is finite and continuous in \(\R^d \setminus \omega\).
\end{definition}

For later use, we observe  that, for every \(t \in \R\),
\begin{equation}
    \label{eqQuasicontinuousOpen}
    \{f > t\} \cup \omega
    \quad \text{is an open set in \(\R^d\),}
\end{equation}
even though \(\{f > t\}\) itself need not be open.
Indeed, by continuity of \(f|_{\R^d \setminus \omega}\) the set \(\{f > t\} \setminus \omega\) is open in \(\R^d \setminus \omega\) with respect to the relative topology.
Hence, there exists an open set \(U\) in \(\R^d\) such that \(\{f > t\} \setminus \omega = U \setminus \omega\) and then
\[
\{f > t\} \cup \omega = U \cup \omega
\]
is open in \(\R^d\) as claimed.

We begin by approximating quasicontinuous functions by sequences of bounded continuous functions:

\begin{proposition}
    \label{propositionQuasicontinuousContinuousBounded}
    Suppose that $H$ satisfies \ref{monotone} and \ref{finite_subadditivity}.
    If \(f : \R^d \to [-\infty, \infty]\) is quasicontinuous and \(\int |f| \dif H < \infty\), then there exists a sequence \((f_n)_{n \in \N}\) of bounded continuous functions in \(\R^d\) such that
    \[
    \lim_{n \to \infty}{\int |f_n - f| \dif H}
    = 0.
    \]
\end{proposition}

\begin{proof}
Suppose that $f$ is quasicontinuous and has finite Choquet integral.
We compose \(f\) with the truncation function 
$T_k : [-\infty, \infty] \to \R$ at height $k > 0$ defined by
\begin{equation}
    \label{eqTruncation}
T_k(t)
= \begin{cases}
k & \text{if \(t > k\),}\\
t & \text{if \(-k \le t \le k\),}\\
-k & \text{if \(t \le -k\).}
\end{cases}
\end{equation}
Then, the composition \(T_k(f)\) satisfies
\begin{align*}
\int |T_k(f) - f| \dif H  
= \int_0^\infty H \bigl( \{ |T_k(f) - f|>t\} \bigr) \dif t 
&= \int_0^\infty H( \{ |f|>k+t\}) \dif t \\
&= \int_k^\infty H( \{ |f|>s\}) \dif s.
\end{align*}
Since $ \int |f| \dif H < \infty$, the integral in the right-hand side tends to zero as $k \to \infty$.  
Next, by quasicontinuity of $f$, for every $\epsilon>0$ we may find an open set $\omega \subset \R^d$ such that $H(\omega) \le \epsilon$ and $f|_{\R^d \setminus \omega}$ is continuous on the closed set \(\R^d \setminus \omega\).  
Then, by composition, $T_k(f)|_{\R^d \setminus \omega}$ is also continuous.  
Thus, the Tietze Extention Theorem allows us to extend $T_k(f)|_{\R^d \setminus \omega}$ as a bounded continuous function $g_{k, \epsilon} : \R^d \to \R$ with \(|g_{k, \epsilon}| \le k\). 
We can therefore estimate using the quasi-sublinearity \eqref{eqQuasiSublinearity} of the Choquet integral,
\begin{align*}
\int |g_{k, \epsilon} - f| \dif H
\le 2 \int |g_{k, \epsilon} - T_k(f)| \dif H + 2 \int |T_k(f) - f| \dif H.
\end{align*}
Taking $n \in \N$, we can find $k = k_n > 0$ such that the second term is bounded by $1/(n+1)$.  
For the first term, since \(|g_{k_n, \epsilon} - T_{k_n}(f)|\) is bounded by \(2k_n\) and vanishes on \(\R^d \setminus \omega\), we have
\begin{align*}
\int |g_{k_n, \epsilon} - T_k(f)| \dif H 
= \int_0^{2k_n} H\bigl( \{ |g_{k_n, \epsilon} - T_{k_n}(f)|>t\} \bigr) \dif t 
\leq \int_0^{2k_n} H(\omega) \dif t
\leq 2k_n \epsilon.
\end{align*}
Thus,
\[
\int |g_{k_n, \epsilon} - f| \dif H 
\le 4 k_n \epsilon + \frac{2}{n+1} \epsilon.
\]
Choosing \(\epsilon = \epsilon_n > 0\) so that \(4 k_n \epsilon_n \le 1/(n+1)\), the right-hand side is less than or equal to \(3/(n + 1)\) and we have the conclusion with \(f_n \vcentcolon= g_{k_n, \epsilon_n}\).
\end{proof}

To obtain the approximation by continuous functions with compact support, one needs a vanishing property of the capacity at infinity:

\begin{mydescription}
\item [{\sc evanescence}]  \label{weak_continuity_from_above}  
For every open subset \(U \subset \R^d\) with \(H(U) < \infty\) and every closed subset \(F \subset U\),
    \begin{align*}
        \lim_{r \to \infty}{H(F \setminus B_r)} = 0. 
    \end{align*}
\end{mydescription}

Under this additional assumption, we prove

\begin{proposition}
    \label{propositionApproximationQuasicontinuousCompactSupport}
    Suppose that $H$ satisfies \ref{monotone}, \ref{finite_subadditivity}, and \ref{weak_continuity_from_above}.
    If \(f : \R^d \to [-\infty, \infty]\) is quasicontinuous and \(\int |f| \dif H < \infty\), then there exists a sequence \((\varphi_n)_{n \in \N}\) in \(C_c(\R^d)\) such that
    \[
    \lim_{n \to \infty}{\int |\varphi_n - f| \dif H}
    = 0.
    \]
\end{proposition}

\begin{proof}
    By Proposition~\ref{propositionQuasicontinuousContinuousBounded} and quasi-sublinearity \eqref{eqQuasiSublinearity} of the Choquet integral, 
    we may assume that \(f\) is bounded and continuous.
    Given \(r > 0\), take \(\psi_r \in C_c(\R^d)\) such that \(0 \le \psi_r \le 1\) in \(\R^d\) and \(\psi_r = 1\) in \(B_r\).
    Since \(|f\psi_r - f| \le |f|\) in \(\R^d\) and \(|f\psi_r - f|= 0\) in \(B_r\), for every \(t > 0\) we have
    \begin{equation}
    \label{eq1305}
    \{|f\psi_r - f| > t\}
    \subset \{|f| \ge t\} \setminus B_r.
    \end{equation}
    By boundedness of \(f\), there exists \(M \ge 0\) such that \(|f| \le M\).
    For any \(0 < \eta \le M\) we then have
    \[
    \int |f\psi_r - f| \dif H
    = \int_0^M H( \{|f\psi_r - f| > t\}) \dif t
    = \int_0^\eta + \int_\eta^M H( \{|f\psi_r - f| > t\} ) \dif t.
    \]
    Using the monotonicity of \(H\) and \eqref{eq1305}, we estimate
    \[
    \int_0^\eta H( \{|f\psi_r - f| > t\}) \dif t
    \le \int_0^\eta H(\{|f| > t\}) \dif t
    \]
    and
    \[
    \int_\eta^M H( \{|f\psi_r - f| > t\} ) \dif t
    \le M H( \{|f\psi_r - f| > \eta\} )
    \le M H\bigl(\{|f| \ge \eta\} \setminus B_r \bigr).
    \]
    Thus,
    \begin{equation}
        \label{eq662}
    \int |f\psi_r - f| \dif H
    \le \int_0^\eta H(\{|f| > t\}) \dif t + M H\bigl(\{|f| \ge \eta\} \setminus B_r \bigr).
    \end{equation}
    Observe that \(\{|f| \ge \eta\}\) is closed and contained in the open set \(\{|f| > \eta\}\), which satisfies \(H(\{|f| > \eta\}) < \infty\) by the Chebyshev inequality.
    Thus, by \ref{weak_continuity_from_above}, the second term in the right-hand side of \eqref{eq662} converges to zero as \(r \to \infty\).
    Since \(\int |f| \dif H < \infty\), the first term in the right-hand side of \eqref{eq662} also tends to zero as \(\eta \to 0\).
    Therefore, if we let $r \to \infty$ and then $\eta \to 0$ in \eqref{eq662}, we obtain
    \begin{align*}
    \lim_{r \to \infty} \int |f\psi_r - f| \dif H =0,
    \end{align*}
    which implies the desired conclusion since \(f \psi_r \in C_c(\R^d)\).
\end{proof}

The assumption \ref{weak_continuity_from_above} is necessary for the density of functions with compact support.
Indeed,

\begin{proposition}
Suppose that \(H\) satisfies \ref{monotone}.
If every quasicontinuous function \(f : \R^d \to \R\) with \(\int |f| \dif H < \infty\) can be approximated in terms of the Choquet integral by a sequence in \(C_c(\R^d)\), then $H$ satisfies \ref{weak_continuity_from_above}.
\end{proposition}

\begin{proof}
    Given an open set \(U \subset \R^d\) with \(H(U) < \infty\) and a closed subset \(F \subset U\), let \(f : \R^d \to \R\) be a continuous function supported in \(U\) such that \(f = 1\) on \(F\) and \(0 \le f \le 1\) in \(\R^d\).
    Then, by monotonicity of \(H\),
    \[
    \int |f| \dif H
    \le H(U) < \infty.
    \]
    By assumption, there exists a sequence \((f_n)_{n \in \N}\) in \(C_c(\R^d)\) such that \(\int |f_n - f| \dif H \to 0\).
    Given \(k \in \N\) to be chosen below, let \(R > 0\) be such that \(\supp{f_k} \subset B_R\).
    By monotonicity of \(H\),  we have
    \[
    \int |f_k - f| \dif H
    \ge \int |f_k - f| \chi_{F \setminus B_R} \dif H
    = \int \chi_{F \setminus B_R} \dif H
    = H(F \setminus B_R).
    \]
    Given \(\epsilon > 0\), take \(k \in \N\) so that the integral in the left-hand side is less than \(\epsilon\). 
    Then, for every \(r \ge R\), we have by monotonicity of \(H\),
    \[
    H(F \setminus B_r) \le H(F \setminus B_R) \le \int |f_k - f| \dif H
    \le \epsilon.
    \qedhere
    \]
\end{proof}

\section{Sublinearity of the Choquet integral}\label{proofs}

The proof of the sublinearity of the Choquet integral in the discrete case relies on the following minimization property:

\begin{lemma}
	\label{lemmaChoquet}
	Suppose that $H$ satisfies \ref{strong_subadditive}.
 Then, for every \(n \in \N_{*}\) and \(C_{1}, \dots, C_{n} \subset \R^d\), there exist \(D_{1}, \dots, D_{n - 1} \subset D_n\)
 such that 
 	\begin{equation*}
	\sum_{i = 1}^{n}{\chi_{D_{i}}}
	= \sum_{i = 1}^{n}{\chi_{C_{i}}}
	\quad
	\text{and}
	\quad
	\sum_{i = 1}^{n}{H(D_{i})}
	\le \sum_{i = 1}^{n}{H(C_{i})}.
	\end{equation*}
\end{lemma}

From the identity between the characteristic functions, we have \(\bigcup\limits_{i = 1}^n{D_i} = \bigcup\limits_{i = 1}^n{C_i}\)
and, since \(D_n\) contains all sets \(D_i\), it then follows that
\(D_n = \bigcup\limits_{i = 1}^n{C_i}\).

\begin{proof}[Proof of Lemma~\ref{lemmaChoquet}]
	We proceed by induction on \(n\).{}
	For \(n = 1\), it suffices to take \(D_{1} = C_{1}\).{}
	We now suppose the statement is true for some \(n \in \N_{*}\).{}
 	Assume we are given sets \(C_{1}\), \dots, \(C_{n+1} \subset \R^d\), and apply the induction assumption to the first \(n\) sets \(C_{1}\), \dots, \(C_{n}\) to get \(\widetilde D_{1}\), \dots, \(\widetilde D_{n - 1} \subset \widetilde D_n\) that satisfy the conclusion.
    By strong subadditivity of \(H\), we have
 	\begin{equation*}
	H(D_{n} \cap C_{n+1}) + H(D_{n} \cup C_{n+1})
 	\le H(D_{n}) + H(C_{n+1}).
	\end{equation*}
    Let \(D_i \vcentcolon= \widetilde D_i\) for \(i \in \{1, \ldots, n-1\}\), \(D_n \vcentcolon= D_{n} \cap C_{n+1}\) and \(D_{n+1} \vcentcolon= D_{n} \cup C_{n+1}\).
 	Then,
 	\begin{align*}
 	\sum_{i = 1}^{n + 1}{H(D_{i})}
	& = \sum_{i = 1}^{n - 1}{H(\widetilde D_{i})} + H(\widetilde D_{n} \cap C_{n+1}) + H(\widetilde D_{n} \cup C_{n+1})\\
 	  & \le \sum_{i = 1}^{n - 1}{H(\widetilde D_{i})} + H(\widetilde D_{n}) + H(C_{n+1}) 
 	 = \sum_{i = 1}^{n}{H(\widetilde D_{i})} + H(C_{n+1}).
	\end{align*}
	By an application of the induction hypothesis, we then have
	\[{}
	\sum_{i = 1}^{n + 1}{H(D_{i})}
	\le \sum_{i = 1}^{n}{H(C_{i})} + H(C_{n+1})
 	= \sum_{i = 1}^{n + 1}{H(C_{i})}.
	\]
	We also observe that
	\[{}
	\chi_{\widetilde D_{n} \cap C_{n + 1}} + \chi_{\widetilde D_{n} \cup C_{n + 1}}
	= \chi_{\widetilde D_{n}} + \chi_{C_{n + 1}}.
	\]
	Since \(\sum\limits_{i = 1}^{n}{\chi_{\widetilde D_{i}}}
	= \sum\limits_{i = 1}^{n}{\chi_{C_{i}}}\), one sees that
	\[{}
    \sum_{i = 1}^{n + 1}{\chi_{D_{i}}}
	 = \sum_{i = 1}^{n - 1}{\chi_{\widetilde D_{i}}} + \chi_{\widetilde D_{n} \cap C_{n+1}} + \chi_{\widetilde D_{n} \cup C_{n+1}}
	 = \sum_{i = 1}^{n}{\chi_{\widetilde D_{i}}} + \chi_{C_{n + 1}}
	= \sum_{i = 1}^{n + 1}{\chi_{C_{i}}},
	\]
 as claimed.  Finally, it remains to verify that \(D_i \subset D_{n+1}\) for \(i \in \{1,\ldots,n\}\).  This follows from the choice of
    \(D_{n + 1} = \widetilde D_{n} \cup C_{n+1}\).
    Indeed, the first \(n - 1\) sets \(D_1, \ldots, D_{n-1}\) are contained in \(\widetilde D_{n}\), which is a subset of \(D_{n+1}\).
    To conclude, one notes that \(D_n = \widetilde D_{n} \cap C_{n+1}\), which is also contained in \(D_{n+1}\).
\end{proof}

The following proposition is claimed true by an induction argument at the top of p.~249 of \cite{Anger:1977}, while a detailed proof can be found on pp.~766--768 of \cite{STW}.
We rely on a different organization of the proof based on Lemma~\ref{lemmaChoquet}.

\begin{proposition}\label{algebraic}
	\label{discrete}
	Suppose that $H$ satisfies \ref{strong_subadditive}.
Then, for every \(C_{1}, \dots, C_{n} \subset \R^d\), the nested family of sets \(A_1 \subset \ldots \subset A_n \subset \R^d\) such that 
\[
\sum\limits_{i = 1}^n{\chi_{A_i}} = \sum\limits_{i = 1}^n{\chi_{C_i}}
\]
satisfies
	\[{}
	\sum_{i = 1}^{n}{H(A_i)}
	\le \sum_{i = 1}^{n}{H(C_{i})}.
	\]
\end{proposition}

For every \(i \in \{1, \ldots, n\}\), the set \(A_i\) is the collection of points that belong to at least \(n - i + 1\) sets among \(C_{1}, \dots, C_{n}\).
Note that \(A_i\) can be also identified as the superlevel set \(\{a \ge i\}\) of the function \(a \vcentcolon= \sum\limits_{i = 1}^{n}{\chi_{C_{i}}}\).

\begin{proof}[Proof of Proposition~\ref{algebraic}]
	  We proceed by induction on the number of sets \(C_{1}, \dots, C_{n}\).
    The conclusion is trivially true for \(n = 1\) since in this case \(A_1 = C_1\).
    We now assume that \(n \ge 2\) and that the statement holds for any family of \(n - 1\) sets.
    Then, given \(C_1, \ldots, C_n \subset \R^d\), take 
    \begin{equation}
    \label{eqChoquet-813}
    D_1, \ldots, D_{n-1} \subset D_n
    \end{equation} 
    given by Lemma~\ref{lemmaChoquet}, 
    so that
	\begin{equation}
	\label{eqChoquet-443}
	\sum_{i = 1}^{n }{\chi_{D_i}} 
    =  
	\sum_{i = 1}^{n}{\chi_{C_i}}
	\quad \text{and} \quad
	\sum_{i = 1}^{n }{H(D_i)}  
	\le  \sum_{i = 1}^{n}{H(C_{i})}.
	\end{equation}
    By the induction assumption applied to \(D_1, \ldots, D_{n-1}\), we have a family of nested sets \(A_1 \subset \ldots \subset A_{n - 1} \subset \R^d\) such that
    \begin{equation}
    \label{eqChoquet-840}
    \sum\limits_{i = 1}^{n - 1}{\chi_{A_{i}}} = \sum\limits_{i = 1}^{n - 1}{\chi_{D_{i}}}
    \quad \text{and} \quad
	\sum_{i = 1}^{n - 1}{H(A_i)}
	\le \sum_{i = 1}^{n - 1}{H(D_i)}.
 \end{equation}
    In particular, \(A_{n-1} = \bigcup\limits_{i = 1}^{n-1}{D_i}\).
    Define \(A_n \vcentcolon= D_n\).
    From \eqref{eqChoquet-813}, we deduce that
    \(A_1 \subset \ldots \subset A_{n - 1} \subset A_n\) are nested.
    Moreover, by \eqref{eqChoquet-443} and \eqref{eqChoquet-840}, they satisfy
    \[
    \sum_{i = 1}^{n}{\chi_{A_i}}
    = \sum_{i = 1}^n{\chi_{D_i}}
    = \sum_{i = 1}^{n }{\chi_{C_i}}
    \]  
    and also
    \[
    \sum_{i = 1}^{n}{H(A_i)}
    \le
    \sum_{i = 1}^{n}{H(D_i)}  
	\le  \sum_{i = 1}^{n}{H(C_{i})}.
    \]
    By induction, the conclusion then follows.
\end{proof}

From this result, sublinearity of the Choquet integral where the functions in consideration have values in a discrete set follows easily:

\begin{corollary}
\label{corollaryChoquetDiscrete}
Suppose that \(H\) satisfies \ref{monotone} and \ref{strong_subadditive}.
If $f,g: \mathbb{R}^d \to \mathbb{N}/k$ for some $k \in \mathbb{N}_{*}$, then
\begin{equation}
	\label{eqChoquet-490}
\int (f+g)\dif H \leq \int f\dif H+ \int g\dif H.
\end{equation}
\end{corollary} 

\begin{proof}
By positive $1$-homogeneity of the Choquet integral we may assume that \(f\) and \(g\) take values in $\mathbb{N}$.
Let us first assume they are both bounded from above by \(n\), in which case
\[
f = \sum_{i = 1}^{n}{\chi_{\{f \ge i\}}}
\quad \text{and} \quad
\int f \dif H
= \sum_{i = 1}^{n}{H(\{f \ge i\})},
\]
with analogous identities for \(g\) and \(f + g\); note that this last function is bounded from above by \(2n\).
Since 
\[
\sum_{i=1}^{2n}{\chi_{\{f + g \ge i\}}}
= 
f+g = \sum_{i=1}^n{\chi_{\{f \ge i\}}} + \sum_{i=1}^n{\chi_{\{g \ge i\}}},
\]
and the sets \(\{f + g \ge i\}\) for \(i \in \{1, \ldots, 2n\}\) are nested,
we may apply Proposition~\ref{discrete} with sets \(\{f \ge i\}\) and \(\{g \ge i\}\) for \(i \in \{1, \ldots, n\}\) to get
\[{}
\sum_{i=1}^{2n}{H(\{f + g \ge i\})} 
\le \sum_{i=1}^{n}{H(\{f \ge i\})} + \sum_{i=1}^{n}{H(\{g \ge i\})}.
\]
From the identities verified by the Choquet integrals of \(f+g\), \(f\) and \(g\), this inequality is precisely \eqref{eqChoquet-490}.

To handle the case where $f$ or $g$ is unbounded, we compose them with the truncation function $T_n$ at height $n \in \mathbb{N}$ defined by
\eqref{eqTruncation}.
As $T_n(f)$ and $T_n(g)$ are both bounded and have values in $\mathbb{N}$, we can apply the sublinearity of the Choquet integral we already proved in this setting to deduce that
\[
\int \bigl(T_n(f) + T_n(g)\bigr) \dif H 
\leq \int T_n(f) \dif H + \int T_n(g) \dif H.
\]
Since \(T_{n}(f) \le f\) and \(T_{n}(g) \le g\), by monotonicity of the Choquet integral we get
\[
\int \bigl(T_n(f) + T_n(g)\bigr) \dif H 
\leq \int f \dif H + \int g \dif H.
\]
We also note that \(T_n(f+g) \leq T_n(f) + T_n(g)\), which again by monotonicity of the Choquet integral gives
\begin{equation}
\label{eqChoquet-537}
\int T_n(f+ g) \dif H 
\leq \int f \dif H + \int g \dif H.
\end{equation}
Since
\[{}
\int_0^n H(\{ f+g>t\}) \dif t 
= \int T_n(f+g) \dif H,
\]
it follows from Fatou's lemma for the Lebesgue integral that
\begin{equation}
\label{eqChoquet-548}
\int (f + g) \dif H
= \int_0^\infty H(\{ f+g>t\}) \dif t 
\le \liminf_{n \to \infty}{\int T_n(f+g) \dif H}.
\end{equation}
It now suffices to combine \eqref{eqChoquet-537} and \eqref{eqChoquet-548}.
\end{proof}

\begin{proof}[Proof of Theorem \ref{sublinear}]
Given \(E, F \subset \R^d\), we have
\[
H(E \cap F) + H(E \cup F)
=\int (\chi_E + \chi_F) \dif H.
\]
Thus, if the Choquet integral is sublinear, then we have
\[
H(E \cap F) + H(E \cup F)
\le \int \chi_E \dif H + \int \chi_F \dif H
= H(E) + H(F).
\]
Hence,  \(H\) is strongly subadditive.

To prove the converse, we now assume that \(H\) is strongly subadditive.
Denote by $\lfloor \alpha \rfloor$ the integer part of the real number \(\alpha\).
Let $f,g:\mathbb{R}^d \to [0,\infty]$ and \(k \in \N_{*}\).
An application of Corollary~\ref{corollaryChoquetDiscrete} to the \(\N/k\)-valued functions $\lfloor k f\rfloor /k$ and $\lfloor k g \rfloor/k$ gives
\[{}
\int \biggl( \frac{\lfloor k f\rfloor}{k} + \frac{\lfloor k g \rfloor}{k} \biggr) \dif H 
 \leq \int \frac{\lfloor k f\rfloor}{k} \dif H + \int \frac{\lfloor k g \rfloor}{k} \dif H.
\]
Since $\lfloor k f\rfloor /k \le f$ and $\lfloor k g \rfloor/k \le g$, by monotonicity of the Choquet integral we then have
\begin{equation}
\label{eqChoquet-562}
\int \biggl( \frac{\lfloor k f\rfloor}{k} + \frac{\lfloor k g \rfloor}{k} \biggr) \dif H 
 \leq \int f \dif H +\int g \dif H.
\end{equation}
Since \(
k \alpha - 1 \leq \lfloor k \alpha \rfloor\) for every \(\alpha \ge 0\), we have
\begin{align*}
f+g -\frac{2}{k}\leq \frac{\lfloor k f\rfloor}{k} + \frac{\lfloor k g \rfloor}{k},
\end{align*}
which implies
\begin{align*}
\biggl(f+g -\frac{2}{k}\biggr)^{+} 
\le \frac{\lfloor k f\rfloor}{k} + \frac{\lfloor k g \rfloor}{k}.
\end{align*}
By definition and monotonicity of the Choquet integral,
\begin{equation}
\label{eqChoquet-586}
\int_{2/k}^\infty H(\{ f+g>t\}) \dif t 
 = 
\int \biggl(f+g -\frac{2}{k}\bigg)^{+} \dif H 
 \leq \int \biggl( \frac{\lfloor k f\rfloor}{k} + \frac{\lfloor k g \rfloor}{k} \biggr) \dif H.
\end{equation}
Hence, by \eqref{eqChoquet-562} and \eqref{eqChoquet-586},
\[{}
\int_{2/k}^\infty H(\{ f+g>t\}) \dif t 
\le \int f \dif H +\int g \dif H.
\]
An application of Fatou's lemma for the Lebesgue integral then yields
\[{}
\int (f + g) \dif H
\le \liminf_{k \to \infty}{\int_{2/k}^\infty H(\{ f+g>t\}) \dif t}
\le \int f \dif H +\int g \dif H.
 \qedhere
\]
\end{proof}

As a consequence of Proposition~\ref{propositionUniformFatou}, one gets countable sublinearity for the Choquet integral for series of functions that converge quasi-uniformly.
Note that countable subadditivity of \(H\) is not necessary in this case.

\begin{corollary}
    \label{propositionUniformTriangleInequalityCountable}
    Suppose that \(H\) satisfies \ref{monotone} and \ref{strong_subadditive}.
    If \((f_n)_{n \in \N}\) is a sequence of real-valued functions in \(\R^d\) such that \(\Bigl(  \sum\limits_{n = 0}^k f_n\Bigr)_{k \in \N}\) converges q.u.\@ to \(F : \R^d \to \R\), then
    \[
    \int |F| \dif H
    \le \sum_{n = 0}^{\infty}{\int |f_n| \dif H}.
    \]
\end{corollary}

\begin{proof}
By monotonicity and sublinearity of the Choquet integral, for every \(k \in \N\) we have
\[
\int \Bigl| \sum_{n = 0}^k{f_n} \Bigr| \dif H
\le \int \sum_{n = 0}^k |f_n| \dif H
\le \sum_{n = 0}^k \int  |f_n| \dif H.
\]
By Proposition~\ref{propositionUniformFatou}, we have the conclusion as \(k \to \infty\).
\end{proof}

\begin{remark}
The statements above and their proofs in this section apply without change to a set function \(H\) defined on an algebra \(\mathcal{X}\) associated to a set \(X\), that is, \(\mathcal{X}\) is a family of subsets of \(X\) such that \(\emptyset \in \mathcal{X}\) and
\[{}
A\cap B, A \cup B \in \mathcal{X} \quad \text{for every \(A, B \in \mathcal{X}\).}
\]
More generally, one expects that results concerning the Choquet integral stated for capacities satisfying \ref{finite_subadditivity} hold in the setting of an algebra, provided one works with the class of $H$-capacitable functions, i.e.\@ those for which their upper-level sets are elements in this algebra.
\end{remark}

\section{Lower semicontinuity of the Choquet integral}\label{lowersemicontinuity}

To obtain a Fatou's lemma without quasi-uniform convergence, we restrict ourselves to the setting of quasicontinuous functions.
The main tool is based on the following standard consequence of the Hahn-Banach Theorem:

\begin{proposition}
    \label{propositionHahnBanach}
    Suppose that $H$ satisfies \ref{monotone}, \ref{strong_subadditive}, and \ref{weak_continuity_from_above}.
    If \(f : \R^d \to [0, \infty]\) is a quasicontinuous function such that \begin{equation}
    \label{eqLevelSets}
    H(\{f > t\}) < \infty
    \quad \text{for every \(t > 0\),}
    \end{equation}
    then
	\[{} 
	\int f \dif H
	= \sup{\left\{\int f \dif \mu
	\left|
	\begin{alignedat}{2}
	& \text{ \(\mu \ge 0\) is a locally finite Borel measure,}\\
	& \text{ \(\mu \le H\) on open subsets of \(\R^d\)}
	\end{alignedat}
	\right.
	\right\}}.
	\]
\end{proposition}

We handle separately the inequality ``\(\ge\)'' in the following

\begin{lemma}
\label{lemmaHahnBanachInequality}
    Suppose that $H$ satisfies \ref{monotone} and \ref{finite_subadditivity}.
If \(f : \R^d \to [0, \infty]\) is  quasicontinuous, then, for every locally finite nonnegative Borel measure \(\mu\) in \(\R^d\) such that \(\mu \le H\) on open subsets of \(\R^d\), we have 
\[
\int f \dif \mu \le \int f \dif H.
\]
\end{lemma}

\begin{proof}[Proof of Lemma~\ref{lemmaHahnBanachInequality}]
Since \(f\) is quasicontinuous, for every \(\epsilon > 0\) there exists an open set \(\omega \subset \R^d\) such that \(H(\omega) \le \epsilon\) and \(f|_{\R^d \setminus \omega}\) is continuous.
Then, by \eqref{eqQuasicontinuousOpen},  the set \(\{f > t\} \cup \omega\) is open in \(\R^d\) for every \(t > 0\).
Since \(\mu\) is monotone, \(\mu \le H\) on open sets, and \(H\) is finitely subadditive, we get
\[
\mu(\{f > t\})
\le \mu(\{f > t\} \cup \omega)
\le H(\{f > t\} \cup \omega)
\le H(\{f > t\}) + H(\omega).
\]
Given \(k > 0\), we integrate both members with respect to \(t \in (0, k)\) to get
\[
\int_0^k \mu(\{f > t\}) \dif t
\le \int_0^k H(\{f > t\}) \dif t + H(\omega) k
\le \int f \dif H + \epsilon k.
\]
Letting \(\epsilon \to 0\) and then \(k \to \infty\), we have the conclusion using Cavalieri's principle for the Lebesgue integral.
\end{proof}

\begin{proof}[Proof of Proposition~\ref{propositionHahnBanach}]
	Inequality ``\(\ge\)'' in the statement readily follows from Lemma~\ref{lemmaHahnBanachInequality} for any nonnegative quasicontinuous function.
 We thus focus on the proof of the reverse inequality 
 ``\(\le\)''.
 We first show it for a nonnegative function \(f \in C_c(\R^d)\), assuming in addition that \(H\) is finite.
	Since \(H\) is strongly subadditive, the function
	\[{}
	P : \psi \in C_{c}(\R^d) \longmapsto \int \psi^{+} \dif H \in \R_{+}
	\]
	is well-defined and sublinear.
	By the Hahn-Banach Theorem, the linear functional
	\[{}
	t f \longmapsto t \int f \dif H
	\]
	defined on the one dimensional vector subspace \(\{t f : t \in \R\}\) has a linear extension \(F : C_{c}(\R^d) \to \R\) such that \(F \le P\) on \(C_{c}(\R^d)\).{}
	
	Observe that if \(\psi \in C_{c}(\R^d)\) is nonpositive, we have \(F(\psi) \le P(\psi) = 0\).{}
	Thus, \(F\) is positive and then, for every compact subset \(S \subset \R^d\), its restriction to \(C(S)\) is a continuous linear functional.{}
	Hence, by the Riesz Representation Theorem, there exists a locally finite nonnegative Borel measure \(\mu\) in \(\R^d\) such that
	\[{}
	F(\psi){}
	= \int \psi \dif \mu{}
	\quad \text{for every \(\psi \in C_{c}(\R^d)\).}
	\]
	In particular, since \(f\) is nonnegative,
	\[{}
	\int f \dif \mu 
	= F(f){}
	= P(f){}
	= \int f \dif H.
	\]
	It remains to observe that \(\mu \le H\) on open subsets of \(\R^d\).{}
	To this end, take a nonempty open set \(U \subset \R^d\) and a compact subset \(K \subset U\).{}
	There exists \(\psi \in C_{c}(\R^d)\) such that \(0 \le \psi \le 1\) in \(\R^d\), \(\psi = 1\) on \(K\) and \(\supp{\psi} \subset U\).{}
	We then have by monotonicity of \(\mu\) and \(H\),
	\[{}
	\mu(K) 
	\le \int \psi \dif\mu{}
	= F(\psi) 
	\le P(\psi)
	= \int \psi \dif H
	\le H(U).
	\]
	Since this inequality holds for every compact subset \(K \subset U\), by inner regularity of \(\mu\) we conclude that \(\mu(U) \le H(U)\).
	
	We next prove the inequality ``\(\le\)'' for a nonnegative function \(f \in C_c(\R^d)\) that satisfies \eqref{eqLevelSets}.
 Take \(\eta > 0\) to be chosen below and consider the set function \(H_\eta\) defined by contraction for every \(A \subset \R^d\) as
 \[
 H_\eta (A)
	\vcentcolon= H(A \cap \{f > \eta\}).
\]
Observe that \(H_\eta\) is also monotone and strongly subadditive.
Moreover, for every \(n \in \N\),
	\begin{equation}
	\label{eqChoquet-1147}
	\int f \dif H_\eta
	 = \int_0^\infty H(\{f > \max{\{t, \eta\}}\}) \dif t
	 \ge \int_\eta^\infty H(\{f > t\}) \dif t.
	\end{equation}
	We may apply the previous case with finite set function \(H_\eta\) to get
    \begin{equation}
    \label{eqChoquet-1163}
    \int f \dif H_\eta
	\le \sup_{\mu \le H_\eta}{\int f \dif \mu}
	\le \sup_{\mu \le H}{\int f \dif \mu}.
    \end{equation}
	Combining \eqref{eqChoquet-1147} and \eqref{eqChoquet-1163}, we then obtain
	\[
    \int_\eta^\infty H(\{f > t\}) \dif t
    \le \sup_{\mu \le H}{\int f \dif \mu}.
	\]
	As \(\eta \to 0\), it follows from Fatou's lemma for the Lebesgue integral that
	\begin{equation}
    \label{eqChoquet-1216}
	\int f \dif H
	\le \sup_{\mu \le H}{\int f \dif \mu}.
	\end{equation}
	
	To conclude, we now prove \eqref{eqChoquet-1216} for an arbitrary nonnegative quasicontinuous function \(f\) for which \eqref{eqLevelSets} holds.
    To this end, we apply Proposition~\ref{propositionQuasicontinuousContinuousBounded} that ensures the existence of a sequence of nonnegative functions \((\varphi_n)_{n \in \N}\) in \(C_c(\R^d)\) such that \(\int |\varphi_n - f| \dif H \to 0\).
	For every \(n \in \N\), we have by sublinearity of the Choquet integral and by Lemma~\ref{lemmaHahnBanachInequality} applied to the quasicontinuous function \(|\varphi_n - f|\),
	\[
	\int \varphi_n \dif \mu
	\le \int f \dif \mu + \int |\varphi_n - f| \dif \mu
	\le \int f \dif \mu + \int |\varphi_n - f| \dif H.
	\]
	Taking the supremum on both sides with respect to \(\mu \le H\), we have by the first part of the proof,
	\[
	\int \varphi_n \dif H
	\le \sup_{\mu \le H}{\int f \dif \mu} + \int |\varphi_n - f| \dif H.
	\]
	As \(n \to \infty\), by convergence of the sequence \((\varphi_n)_{n \in \N}\) with respect to the Choquet integral we get \eqref{eqChoquet-1216}, which completes the proof.
\end{proof}

Assumption \eqref{eqLevelSets} can be removed under inner regularity of open sets with infinite capacity:

\begin{mydescription}
\item[{\sc semifinite}]
\label{regularity_infinite_capacity} 
For every open set \(U \subset \R^d\) such that \(H(U) = \infty\) and every \(M \ge 0\), there exists a subset \(E \subset U\) with \(M \le H(E) < \infty\).
\end{mydescription}

We then get the counterpart of Proposition~\ref{propositionHahnBanach} for all nonnegative quasicontinuous functions with infinite Choquet integral:

\begin{proposition}
    \label{propositionHahnBanachInfinite}
    Suppose that $H$ satisfies \ref{monotone}, \ref{strong_subadditive}, \ref{weak_continuity_from_above}, and \ref{regularity_infinite_capacity}.
    If \(f : \R^d \to [0, \infty]\) is a quasicontinuous function such that
    \begin{equation}
    \label{eqLevelSetFalse}
    H(\{f > t\}) = \infty
    \quad \text{for some \(t > 0\),}
    \end{equation}
    then 
\[
\int f \dif H = \infty = 
\sup{\left\{\int f \dif \mu
	\left|
	\begin{alignedat}{2}
	& \text{ \(\mu \ge 0\) is a locally finite Borel measure,}\\
	& \text{ \(\mu \le H\) on open subsets of \(\R^d\)}
	\end{alignedat}
	\right.
	\right\}}.
	\]
\end{proposition}

\begin{proof}
    From \eqref{eqLevelSetFalse}, we have \(\int f \dif H = \infty\).
    Next, by quasicontinuity of \(f\), there exists an open set \(\omega \subset \R^d\) such that \(H(\omega) \le 1\) and \(f|_{\R^d \setminus \omega}\) is continuous.
    In particular, by \eqref{eqQuasicontinuousOpen}, the set \(\{f > t\} \cup \omega\) is open in \(\R^d\).
    By monotonicity of \(H\) and \eqref{eqLevelSetFalse}, 
    \[
    H(\{f > t\} \cup \omega) \ge H(\{f > t\}) = \infty.
    \]
    Applying \ref{regularity_infinite_capacity}, for every \(n \in \N\) there exists a 
    subset \(E_n \subset \{f > t\} \cup \omega\) with \(n \le H(E_n) < \infty\).
    Let \(H_n\) be the set function defined for every \(A \subset \R^d\) by \(H_n(A) \vcentcolon= H(A \cap E_n)\).
    Note that
    \[
    t H(\{f > t\} \cap E_n )
    \le \int_0^t H_n(\{f > s\}) \dif s
    \le \int f \dif H_n.
    \]
    Since \(E_n \subset \{f > t\} \cup \omega\), by monotonicity and finite subadditivity of \(H\),
    \[
    n \le H(E_n)
    \le H(\{f > t\} \cap E_n )
    + H(\omega)
    \le H(\{f > t\} \cap E_n )
    + 1
    \]
    Thus,
    \begin{equation}
    \label{eq1429}
    t (n - 1)
    \le \int f \dif H_n
    .
    \end{equation}
    Since \(H_n\) is finite and \(H_n \le H\), by Proposition~\ref{propositionHahnBanach} we have
    \begin{equation}
    \label{eq1437}
    \int f \dif H_n 
    \le \sup_{\mu \le H_n}{\int f \dif \mu} 
    \le \sup_{\mu \le H}{\int f \dif \mu}.
    \end{equation}
    Combining \eqref{eq1429} and \eqref{eq1437}, we get
    \[
    t(n - 1)
    \le \sup_{\mu \le H}{\int f \dif \mu}.
    \]
    This implies the conclusion since \(n\) can be chosen arbitrarily large.
\end{proof}

From Propositions~\ref{propositionHahnBanach} and~\ref{propositionHahnBanachInfinite}, we deduce the following analogue of Fatou's lemma for quasicontinuous functions:

\begin{corollary}
\label{corollaryFatou}
    Suppose that $H$ satisfies \ref{monotone}, \ref{strong_subadditive}, \ref{weak_continuity_from_above}, and \ref{regularity_infinite_capacity}.
If \((f_n)_{n \in \N}\) is a sequence of nonnegative quasicontinuous functions in \(\R^d\),
then, for every quasicontinuous function \(f : \R^d \to [0, \infty]\) such that 
\(f \le \liminf\limits_{n \to \infty}{f_n}\) in \(\R^d\), we have
\[
\int f \dif H
\le \liminf_{n \to \infty}{\int f_n \dif H}.
\]
\end{corollary}

\begin{proof}
We may assume that each \(f_n\) has finite Choquet integral and that \(\liminf\limits_{n \to \infty}{\int f_n \dif H} < \infty\).
For every nonnegative locally finite Borel measure \(\mu\) in \(\R^d\) such that \(\mu \le H\) on open subsets of \(\R^d\), we have by monotonicity  and Fatou's lemma for the Lebesgue integral,
\[
\int f \dif\mu 
\le \int \liminf_{n \to \infty}{f_n} \dif\mu
\le \liminf_{n \to \infty}{\int f_n \dif \mu}.
\]
Thus, by Lemma~\ref{lemmaHahnBanachInequality},
\begin{equation}
\label{eq-1311}
\int f \dif\mu
\le \liminf_{n \to \infty}{\int f_n \dif H}.
\end{equation}
Since \(f\) is quasicontinuous and the right-hand side is finite, by Proposition~\ref{propositionHahnBanachInfinite} we must have \(H(\{f > t\}) < \infty\) for every \(t > 0\).
The assumptions of Proposition~\ref{propositionHahnBanach} are then satisfied by \(f\) and it thus suffices to take the supremum with respect to \(\mu\) in the left-hand side of \eqref{eq-1311}.
\end{proof}

The assumption \ref{regularity_infinite_capacity} avoids a gap between sets of finite and infinite capacities, which is illustrated in the following example where the conclusion of Corollary~\ref{corollaryFatou} fails:

\begin{example}
\label{exampleInfinity}
Let \(H\) be defined for every nonempty subset \(A \subset \R^d\) as 
\[
H(A) =
\begin{cases}
1 & \text{if \(A\) is bounded,}\\
\infty & \text{if \(A\) is unbounded,}
\end{cases}
\]
For each \(n \in \N\), take \(f_n \in C_c(\R^d)\) such that \(0 \le f_n \le 1\) in \(\R^d\), \(f_n = 1\) in \(B_n\) and \(\supp{f_n} \subset B_{n + 1}\).
Then, \(\int f_n \dif H = 1\) and \(f_n \to 1\) pointwise in \(\R^d\) as \(n \to \infty\), but \(\int 1 \dif H = \infty\).
\end{example}

One can fulfill assumption \ref{regularity_infinite_capacity} by replacing \(H\) with a more regular set function \(\widetilde{H} : \mathcal{P}(\R^d) \to [0, \infty]\) defined for every \(A \subset \R^d\) by
\[
\widetilde{H}(A)
\vcentcolon= \sup{\Bigl\{ H(D) : D \subset A,\ H(D) < \infty  \Bigr\}}
.
\]
Observe that \(\widetilde{H}\) is monotone regardless of \(H\).
Under monotonicity of \(H\) itself, these set functions coincide on sets of finite \(H\)~capacity:

\begin{proposition}
    If \(H\) satisfies \ref{monotone}, then 
    \[
    \widetilde{H}(A)
    = H(A)
    \quad \text{for every \(A \subset \R^d\) with \(H(A) < \infty\).}
    \]
    Hence, \(\widetilde{H}\) also satisfies \ref{regularity_infinite_capacity}.
\end{proposition}

\begin{proof}
By the monotonicity of \(H\), the supremum in the definition of \(\widetilde{H}(A)\) is achieved by the set \(A\) itself whenever \(H(A) < \infty\).
That \ref{regularity_infinite_capacity} holds, then follows from the fact that we may then reformulate the definition of \(\widetilde{H}\) as
\[
\widetilde{H}(A)
= \sup{\Bigl\{ \widetilde{H}(D) : D \subset A,\ \widetilde{H}(D) < \infty  \Bigr\}}.
\qedhere
\]
\end{proof}

Note that \(\widetilde{H}\) inherits several properties of \(H\):

\begin{proposition}
    Suppose that \(H\) satisfies \ref{monotone}.
    If \(H\) also satisfies any of the assumptions  \ref{finite_subadditivity}, \ref{countable_subadditivity}, \ref{strong_subadditive}, or \ref{weak_continuity_from_above}, then so does \(\widetilde{H}\).
    \comment{The formulation of this statement deserves some improvement.}
\end{proposition}

\begin{proof}
    That \(\widetilde{H}\) verifies \ref{weak_continuity_from_above} whenever \(H\) does follows from the fact that \(\widetilde{H} = H\) on sets where \(H\) is finite.
    We now assume that \(H\) satisfies \ref{countable_subadditivity}.
    Given a sequence of sets \((E_n)_{n \in \N}\) and \(D \subset \bigcup\limits_{n = 0}^\infty{E_n}\) with \(H(D) < \infty\), by monotonicity of \(H\) for every \(n \in \N\) we have \(H(E_n \cap D) < \infty\).
    Hence, by countable subadditivity of \(H\) and definition of \(\widetilde{H}\),
    \[
    H(D) = H\Bigl( \bigcup_{n = 0}^\infty (E_n \cap D) \Bigr)
    \le \sum_{n = 0}^\infty{H(E_n \cap D)}
    \le \sum_{n = 0}^\infty{\widetilde H(E_n)}.
    \]
    It now suffices to take the supremum with respect to \(D\) in the left-hand side.
    Similarly, one verifies that \(\widetilde{H}\) satisfies \ref{finite_subadditivity} whenever \(H\) does. 
    
    We now assume that \(H\) verifies \ref{strong_subadditive}.
    Given \(E, F \subset \R^d\), take \(C \subset E \cap F\) and \(D \subset E \cup F\) with \(H(C) < \infty\) and \(H(D) < \infty\).
    By finite subadditivity of \(H\), \(H(C \cup D) < \infty\) and then, by monotonicity, \(H\) is finite on every subset of \(C \cup D\).
    Note that, since \(C \subset E \cap F\) and \(D \subset E \cup F\),
    \[
    C \subset ((C \cup D) \cap E) \cap ((C \cup D) \cap F)
    \quad \text{and} \quad 
    D \subset ((C \cup D) \cap E) \cup ((C \cup D) \cap F).
    \]
    By monotonicity and strong subadditivity of \(H\), we then have
    \[
    \begin{split}
    H(C) + H(D) 
    & \le H\bigl(((C \cup D) \cap E) \cap ((C \cup D) \cap F)\bigr) + H\bigl(((C \cup D) \cap E) \cup ((C \cup D) \cap F)\bigr)\\
    & \le H\bigl((C \cup D) \cap E\bigr) + H\bigl((C \cup D) \cap F\bigr).
    \end{split}
    \]
    Since \((C \cup D) \cap E\) and \((C \cup D) \cap F\) are subsets of \(E\) and \(F\) where \(H\) is finite, we get by definition of \(\widetilde{H}\),
    \[
    H(C) + H(D) 
    \le \widetilde{H}(E) + \widetilde{H}(F).
    \]
    It now suffices to take the supremum in the left-hand side with respect to \(C\) and \(D\).
\end{proof}

\section{The space $L^1(H)$}\label{Lone}

We assume that \(H\) satisfies \ref{monotone} and \ref{strong_subadditive}.
We introduce an equivalence relation \(\sim\) among elements in the vector space of real-valued quasicontinuous functions in \(\R^d\) by denoting \(f \sim g\) whenever \(f = g\) quasi-everywhere (q.e.), that is,  there exists \(E \subset \R^d\) such that \(H(E) = 0\) and \(f = g\) in \(\R^d \setminus E\).
Then, \(\Class{f}\) is the equivalence class that contains \(f\).
We let
\[
L^1(H) \vcentcolon= \biggl\{ \Class{f} : f : \R^d \to \R\ \text{is quasicontinuous and } \int |f| \dif H < \infty \biggr\}.
\]
We may naturally equip this set with addition and multiplication by scalar: For every quasicontinuous functions \(f, g\) and \(\lambda \in \R\), let
\[
\Class{f} + \Class{g} \vcentcolon= \Class{f + g}
\quad \text{and} \quad
\lambda \Class{f} \vcentcolon = \Class{\lambda f}.
\]
Observe that the function
\[
\| \Class{f} \|_{L^1(H)}\vcentcolon= \int |f | \dif H
\]
is well-defined in \(L^1(H)\) and, by sublinearity and \(1\)-homogeneity of the Choquet integral, is a norm in this space.

\begin{proposition}
\label{propositionL1Banach}
Suppose that $H$ satisfies \ref{monotone}, \ref{strong_subadditive} and \ref{countable_subadditivity}.  
Then, $L^1(H)$ is a Banach space.
\end{proposition}

The proof is standard, e.g.\@ \cite{C5}*{Proposition~2.1 and 2.2}, though as stated in the introduction we include it for completeness.

\begin{proof}
Let $(\Class{f_n})_{n \in \N}$ be a Cauchy sequence in \(L^1(H)\), where each \(f_n : \R^d \to \R\) is quasicontinuous. We may find positive integers $n_1<n_2<\ldots <n_j$ such that
\begin{equation}
\label{eq-1581}
\|\Class{f_m - f_n}\|_{L^1(H)} 
= \|\Class{f_m} - \Class{f_n}\|_{L^1(H)} 
\leq\frac{1}{4^{j}}
\quad \text{for every $m, n \geq n_j$.}
\end{equation}
This implies
\begin{align*}
\int_0^{1/2^{j}} H\bigl( \bigr\{ |f_{n_j} - f_{n_{j+1}}|> 1/2^{j} \bigr\} \bigr) \dif t \leq \frac{1}{4^{j}}
\end{align*}
and therefore
\begin{align*}
H\bigl( \bigl\{ |f_{n_j} - f_{n_{j+1}}| > 1/2^{j} \bigr\} \bigr) \leq \frac{1}{2^{j}}.
\end{align*}
Since \(|f_{n_j} - f_{n_{j+1}}|\) is quasicontinuous, there exists an open set \(\omega_j \subset \R^d\) such that \(H(\omega_j) \le 1/2^j\) and \(|f_{n_j} - f_{n_{j+1}}|\) is continuous in \(\R^d \setminus \omega_j\).
By \eqref{eqQuasicontinuousOpen}, the set
\begin{align*}
G_j \vcentcolon= \bigl\{ |f_{n_j} - f_{n_{j+1}}|>2^{-j} \bigr\} \cup \omega_j
\end{align*}
is open and, by finite subadditivity of \(H\), 
\(H(G_j) \le {1}/{2^{j - 1}}
\).

Let
\(F_m \vcentcolon= \bigcup\limits_{j = m}^\infty G_j
\).
For any $x \in \R^d \setminus F_m$ we have
\begin{align*}
\sum_{l = m}^\infty |f_{n_l}(x)-f_{n_{l+1}}(x)| \leq \sum_{l = m}^\infty \frac{1}{2^{l}} <+\infty.
\end{align*}
Therefore if for $x \in \R^d \setminus F_m$ one defines
\begin{align*}
f(x)
\vcentcolon= \lim_{j \to \infty} f_{n_j}(x) 
= f_{n_1}(x) + \lim_{k \to \infty} \sum_{l=1}^k \bigl(f_{n_{l+1}}(x)-f_{n_l}(x) \bigr),
\end{align*}
then by the Weierstrass $M$-test, 
\((f_{n_j})_{j \in \N}\) converges uniformly to \(f\) in $\R^d \setminus F_m$, whence $f$ is continuous in $\R^d \setminus F_m$.  
As $(F_m)_{m \ge 1}$ is non-increasing, the function $f$ is well-defined on $\bigcup\limits_{m \geq 1}{(\R^d \setminus F_m)}$.
Finally, let \(f(x) = 0\) for \(x \in \bigcap\limits_{m \ge 1}{F_m}\).
Since, by countable subadditivity of \(H\),\begin{align*}
H(F_m) 
\leq \sum_{j = m}^\infty H(G_j) \leq \frac{1}{2^{m-2}} \to 0
\end{align*}
as $m \to \infty$,  
we deduce that \(f\) is quasicontinuous and \(f_{n_j} \to f\) q.u.

Since \((\Class{f_n})_{n \in \N}\) is a Cauchy sequence in \(L^1(H)\), to prove its convergence to \(\Class{f}\) in \(L^1(H)\) it suffices to prove that the subsequence \((\Class{f_{n_j}})_{j \ge 1}\) converges q.u.\@ to \(\Class{f}\).
For every \(i \ge j \ge 1\), by \eqref{eq-1581} we have
\[
\int |f_{n_j} - f_{n_i}| \dif H
\le \frac{1}{4^j}.
\]
As \((|f_{n_j} - f_{n_i}|)_{i \ge 1}\) converges to \(|f_{n_j} - f|\) when \(i \to \infty\), it follows from Proposition~\ref{propositionUniformFatou} that
\[
\int |f_{n_j} - f| \dif H
\le
\liminf_{i \to \infty}{\int |f_{n_j} - f_{n_i}| \dif H}
\le \frac{1}{4^j},
\]
from which the conclusion follows.
\end{proof}

Every function in \(C_c(\R^d)\) has finite Choquet integral provided that \(H\) satisfies
\begin{mydescription}
    \item[{\sc locally finite}] \label{inclusion_continuous_functions}
    $H(U)<+\infty$ for every bounded open set $U \subset \R^d$.
\end{mydescription}
Then, the quotient space \(C_c(\R^d)/ \!\sim\) is contained in \(L^1(H)\).
Moreover,

\begin{corollary}
Suppose that $H$ satisfies \ref{monotone}, \ref{strong_subadditive}, \ref{countable_subadditivity}, \ref{weak_continuity_from_above}, and \ref{inclusion_continuous_functions}.
Then, \(L^1(H)\) is the completion of \(C_c(\R^d)/ \!\sim \;\)  with respect to the \(L^1(H)\) norm.
\end{corollary}

\begin{proof}
From Proposition~\ref{propositionApproximationQuasicontinuousCompactSupport} we have that \(C_c(\R^d)/ \!\sim \;\) is dense in \(L^1(H)\).
By Proposition~\ref{propositionL1Banach}, \(L^1(H)\) is complete.
\end{proof}

We conclude this section with a closure property for bounded sequences in \(L^1(H)\).
Firstly, concerning quasi-uniform convergence:

\begin{corollary}
Suppose that \(H\) satisfies \ref{monotone} and \ref{strong_subadditive}.
If \((\Class{f_n})_{n \in \N}\) is a bounded sequence in \(L^1(H)\) such that \((f_n)_{n \in \N}\) converges q.u.\@ to a quasicontinuous function \(f\), then \(\Class{f} \in L^1(H)\) and
\[
\|\Class{f}\|_{L^1(H)}
\le \liminf_{n \to \infty}{\|\Class{f_n}\|_{L^1(H)}}.
\]
\end{corollary}

\begin{proof}
It suffices to apply Proposition~\ref{propositionUniformFatou}, which implies that \(\int |f| \dif H < \infty\) and then \([f] \in L^1(H)\) by quasicontinuity of \(f\).
\end{proof}

Observe that if \(f, g : \R^d \to \R\) are such that \(f = g\) q.e. and \(f\) is quasicontinuous, it need not be true that \(g\) is quasicontinuous.
To make sure that such a property holds, one may require some regularity on sets where \(H\) vanishes:
\begin{mydescription}
    \item[{\sc zero-capacity regularity}] \label{regularity_zero_capacity}
    For every \(E \subset \R^d\) with \(H(E) = 0\) and every \(\epsilon > 0\), there exists an open set \(\omega \supset E\) such that \(H(\omega) \le \epsilon\).
\end{mydescription}

\begin{corollary}
Suppose that \(H\) satisfies \ref{monotone}, \ref{strong_subadditive},  \ref{weak_continuity_from_above}, \ref{regularity_infinite_capacity}, and \ref{regularity_zero_capacity}.
If \((\Class{f_n})_{n \in \N}\) is a bounded sequence in \(L^1(H)\) such that \((f_n)_{n \in \N}\) converges q.e.\@ to a quasicontinuous function \(f\), then \(\Class{f} \in L^1(H)\) and
\[
\|\Class{f}\|_{L^1(H)}
\le \liminf_{n \to \infty}{\|\Class{f_n}\|_{L^1(H)}}.
\]
\end{corollary}

\begin{proof}
We observe that \((|f_n|)_{n \in \N}\) converges q.e.\@ to \(|f|\) and that \(|f|\) is quasicontinuous.
It then follows from \ref{regularity_zero_capacity} that the function \(\liminf\limits_{n \to \infty}{|f_n|}\), which equals \(|f|\) q.e., is also quasicontinuous.
Moreover, by Corollary~\ref{corollaryFatou}, we deduce that
\[
\int |f| \dif H
= \int \liminf_{n \to \infty}{|f_n|} \dif H
\le \liminf_{n \to \infty}{\int |f_n| \dif H},
\]
from which the conclusion follows.
\end{proof}

\section*{Acknowledgments}
D.~Spector is supported by the National Science and Technology Council of Taiwan under research grant number 110-2115-M-003-020-MY3 and the Taiwan Ministry of Education under the Yushan Fellow Program.  Part of this work was undertaken while D.~Spector was visiting IRMP Institute of the Universit\'e catholique de Louvain.  He would like to thank the IRMP Institute for its support and A.~C.~Ponce for his warm hospitality during the visit.

\end{document}


\bibliographystyle{amsalpha}